\documentclass{article}
\usepackage{spconf,amsmath,graphicx}
\usepackage{amssymb}
\usepackage{amsfonts}
\usepackage{physics}
\usepackage{amsthm}
\usepackage{array}
\usepackage{arydshln}
\usepackage{multirow}
\usepackage{booktabs}
\usepackage{adjustbox}

\newcommand{\hide}[1]{}

\usepackage{algorithm} 
\usepackage{algpseudocode} 

\usepackage{subcaption}
\usepackage{caption}
\usepackage{xcolor}

\title{Curvature accelerated decentralized non-convex optimization for high-dimensional machine learning problems}
\name{Dingran Yi, Fanhao Zeng, and Nikolaos M. Freris\thanks{School of Computer Science and Technology, University of Science and Technology of China, Hefei, Anhui, 230027, China. 
Emails: \texttt{ydr0826@mail.ustc.edu.cn,~sa23011002@mail.ustc.edu.cn, nfr@ustc.edu.cn}.}}
\address{}

\begin{document}
%
\maketitle
\begin{abstract}
We consider distributed optimization as motivated by machine learning in a multi-agent system: each agent holds local data and the goal is to minimize an aggregate loss function over a common model, via an interplay of local training and distributed communication. In the most interesting case of training a neural network, the loss functions are non-convex and the high dimension of the model poses challenges in terms of communication and computation. We propose a primal-dual method that leverages second order information in the local training sub-problems in order to accelerate the algorithm. To ease the computational burden, we invoke a quasi-Newton local solver with linear cost in the model dimension. Besides, our method is communication efficient in the sense of requiring to broadcast the local model only once per round. We rigorously establish the convergence of the algorithm and demonstrate its merits by numerical experiments.
\end{abstract}


\section{introduction}

Distributed optimization has found many applications in fields such as transportation \cite{hajebrahimi2020adaptive}, power systems \cite{molzahn2017survey} and control \cite{nedic2018distributed}. Agents cooperate to minimize a sum of local cost functions:
\begin{equation}
    \underset{x\in\mathbb{R}^d}{\text{minimize}}\qquad \sum_{i=1}^m f_i(x) , \label{prob1}
\end{equation} 
where $f_i(\cdot)$ is the cost function (referred to as loss function in the machine learning literature) held by agent $i$. Distributed optimization algorithms operate as follows: agents store and update local models $x_i$ through a synergy of computation on $f_i$ as well as exchanging information with neighbors so that (\ref{prob1}) is solved asymptotically with $\left\{x_i\right\}$ reaching consensus.

There is an abundance of methods for both the agent-server setting \cite{mcmahan2017communication,9835545,karimireddy2020scaffold} and the serverless setting on a general communication graph \cite{nedic2017achieving,bastianello2020asynchronous}. However, the underlying analysis for serverless setting assumes convexity which is not applicable for the problem we target in this paper, namely training neural networks. Meanwhile, the majority use first order solvers for the corresponding local problems because of the low cost and simple implementation. However, first order methods generally suffer from slow convergence \cite{juditsky5}. There are also several works that invoke curvature information (i.e., Newton-type methods) to attain acceleration \cite{liu2023communication,crane2020dino,li2023communication,zhou2011quasi}. Nonetheless, this leads to a substantial increase of per round computation and communication (as well as storage), which renders them unattractive or impractical for problems where the model dimension $d$ is very large.

We propose a decentralized primal-dual method termed CADEN (Curvature Accelerated DEcentralized Non-convex) that applies to non-convex problems with high dimension. It is based on the communication-efficient implementation of ADMM (Alternating Direction Method of Multipliers) developed in \cite{li2023communication}. However, our method is different in terms of three key aspects: a) it applies to non-convex problems, b) it specifically targets high dimension $d$, and c) it allows variable workload across agents. In specific, we invoke L-BFGS for (inexactly) solving the local sub-problems, which can accelerate the convergence while maintaining acceptable computation cost ($\mathcal{O}(d)$). In contrast, existing methods for non-convex problems either support only first order solvers \cite{alghunaim2019linearly}, or become prohibitively costly when the dimension is large because they require storing or communicating Hessian matrices \cite{mokhtari2016network,tutunov2019distributed}. Besides, all aforementioned methods are synchronous (i.e., they require the participation of all agents in each round) and assume a common workload across agents. As a consequence, they are not suitable for scenarios with pronounced system heterogeneity (i.e., when agents become unavailable due to low energy or bandwidth and also have diverse computing power).\\   
\textbf{Contributions}:
\begin{enumerate}
\item We develop a primal-dual method for high-dimensional non-convex problems with three key attributes. First, it incorporate the curvature in a computationally efficient manner (using L-BFGS as the local solver) to accelerate the convergence. Second, it is communication-efficient in the sense of requiring a single broadcast of the local model per round. Third, it allows partial participation and variable local workload (across agents and rounds).
\item We provide a theoretical analysis that establishes convergence to stationary points with a rate of $\mathcal{O}\left(\frac{1}{T}\right)$, where $T$ is the number of rounds.
Our analysis reveals the acceleration obtained by L-BFGS compared to using a first order local solver. Besides, it characterizes the convergence rate in terms of local work, partial participation, communication topology, and conditioning of the loss functions.
\item Experiments on training a neural network in the multi-agent setting demonstrate advantages in terms of superior convergence rate with noticeable communication and computation savings. 
\end{enumerate}
\section{Proposed Method}
\subsection{Problem reformulation}
We begin by re-formulating (\ref{prob1}) in a means that is suitable for obtaining a distributed method that can incorporate second-order information in the local problems while maintaining communication efficiency \cite{li2023communication}. We consider a connected undirected graph $G=\left\{\mathcal{V},\mathcal{E}\right\}$ where $\mathcal{V}=\left[m\right]:=\{1,\hdots,m\}$ is the set of agents and $\mathcal{E}\subset \mathcal{V}\times\mathcal{V}$ is the set of edges, with $\left(i,j\right)\in\mathcal{E}$ if and only if agent $i$ can communicate with agent $j$. We define the set of neighbors of agent $i$ as $\mathcal{N}_i:=\left\{j| (i,j)\in\mathcal{E}\right\}$. Furthermore, $n:=\lvert\mathcal{E}\rvert$ denotes the number of edges and $d_i:=\lvert\mathcal{N}_i\rvert$ denotes the degree of agent $i$. Problem (\ref{prob1}) is equivalent to
\begin{align}
    \underset{x_i,z_{ij}\in\mathbb{R}^d}{\text{minimize}}\qquad &\sum_{i=1}^m f_i(x_i),\notag\\
    \text{s.t.}\qquad x_i=&z_{ij}=x_j,\,\, j\in\mathcal{N}_i.\label{prob2}
\end{align}
The intermediate variables $\left\{z_{ij}\right\}$ are to attain a block-diagonal Hessian for the augmented Lagrangian, which is key to communication efficiency (i.e., no need for multiple rounds of distributed message passing) when curvature is invoked. This can be compactly written as
\begin{align}
\underset{x\in\mathbb{R}^{md},z\in\mathbb{R}^{nd}}{\text{minimize}}\qquad &F(x),\notag\\
\text{s.t.}\qquad Ax=&Bz,\label{prob3}
\end{align}where $F(x):=\sum_{i=1}^m f_i(x_i)$, $x$ and $z$ concatenate the variables $\left\{x_i\right\},\left\{z_{ij}\right\}$ respectively; $A:=\begin{bmatrix}\hat{A}_s\otimes I_d\\\hat{A}_d\otimes I_d\end{bmatrix}$ (where $\hat{A}_s, \hat{A}_d$ are $n\times m$ binary matrices with $\left[\hat{A}_s\right]_{ki}=\left[\hat{A}_d\right]_{kj}=1$ if and only if $(i,j)$ is the $k$-th edge and $0$ otherwise), and $ B:=\begin{bmatrix}I_{nd}\\I_{nd}\end{bmatrix}$. The augmented Lagrangian for (\ref{prob3}) is
\begin{equation}
\mathcal{L}(x,z,y)=F(x)+y^\top(Ax-Bz)+\frac{\mu_z}{2}\norm{Ax-Bz}^2,\label{AL}
\end{equation}where $y$ is the dual variable and $\mu_z>0$ is the quadratic coefficient. The iterations
of ADMM are given by sequential alternating minimization
of the AL over $x, z$ plus a dual ascent step of $y$. 
\subsection{The proposed method}
The iterations of ADMM can be carried distributively as follows:
\begin{subequations}
\begin{align}
    x_i^{t+1}=\,&\underset{x_i}{\text{arg min}}\,\, \Big\{f_i(x_i)\notag\\&+\sum_{j\in\mathcal{N}_i}\left( (y_{ij,i}^t)^\top (x_i-z_{ij}^t)+\frac{\mu_z}{2}\norm{x_i-z_{ij}^t}^2\right)\Big\},\label{name_x}\\
    z_{ij}^{t+1}=\,&\frac{1}{2}\left(x_i^{t+1}+x_j^{t+1}+\frac{1}{\mu_z}(y_{ij,i}^t+y_{ij,j}^t)\right),\label{name_z}\\y_{ij,i}^{t+1}=\,&y_{ij,i}^{t}+\mu_y(x_i^{t+1}-z_{ij}^{t+1}),\label{name_y}
\end{align}\label{seem5}
\end{subequations}
where $\mu_y>0$ is the dual ascent parameter. Here we adopt a new parameter $\mu_y$ instead of directly using $\mu_z$ as standard ADMM to for establishment of theoretical proof. Nevertheless, $y,z-$variables pertain to edges: this not only incurs additional storage costs, but it also makes partial participation problematic (if an agent is active it is expected to update all its adjacent edge variables, however, some of its neighbors may be inactive at this round). For this reason, we proceed to eliminate all edge variables to obtain a method with agent-specific variables only. By proper initialization, i.e., $y_{ij,i}^0+y_{ij,j}^0=0$, we may first conclude by induction that $y_{ij,i}^t+y_{ij,j}^t=0$ for all $t$ and (\ref{name_z}) can be simplified as $z_{ij}^{t}=\frac{1}{2}\left(x_i^{t}+x_j^{t}\right)$. Then, by replacing the expression for $z_{ij}^t$ and defining $\phi_i^t:=\sum_{j\in\mathcal{N}_i}y_{ij,i}^t$, it is easy to inspect that (\ref{name_x}) is equivalent with
\begin{footnotesize}
\begin{align}
    x_i^{t+1}=\underset{x_i}{\text{arg min}}\left(f_i(x_i)+(\phi_i^t)^\top x_i+\frac{\mu_z}{2}\sum_{j\in\mathcal{N}_i}\left\|x_i-\frac{x_i^t+x_j^t}{2}\right\|^2\right).\label{simp_x}
\end{align}
\end{footnotesize}
Finally, from the definition of $\phi_i^t$ and (\ref{name_y}), we obtain the update rule for $\phi$ as $$\phi_i^{t+1}=\phi_i^t+\frac{\mu_y}{2}\sum_{j\in\mathcal{N}_i}\left(x_i^{t+1}-x_{j}^{t+1}\right).$$
\begin{algorithm}[t]
	\caption{CADEN} 
	\textbf{Initialization}: zero initialization for $\phi$
	\begin{algorithmic}[1]
		\For {$t=0,1,2,\hdots$}
  \For {active agent~$i$}
		\Statex \texttt{primal update}:
\State{$x^{t+1}_i=\texttt{lbfgs}\,\left(f_i(x_i)+\left(\phi_i^t\right)^{\top} x_i\right.$}
        \Statex \qquad\qquad\qquad\qquad\qquad$\left.+\frac{\mu_z}{2}\sum_{j\in\mathcal{N}_i}\norm{x_i-\frac{x_i^t+x_j^t}{2}}^2\right)$
        \Statex \texttt{communication}:
        \State{broadcast $x_i^{t+1}$ to neighbors}
        \Statex \texttt{dual update}:
        \State{$\phi_i^{t+1}=\phi_i^t+\frac{\mu_y}{2}\sum_{j\in\mathcal{N}_i}\left(x_i^{t+1}-x_j^{t+1}\right)$}
        \EndFor
		\EndFor 
	\end{algorithmic} 
\end{algorithm}

The algorithmic description is provided in Alg. 1. In CADEN, all variables correspond to agents: $x_i$ represents the local model and $\phi_i$ represents the dual variable for agent $i$. Partial participation is reflected in step~2. In step~3, active agents update (in parallel) their local models $x_i$ by inexactly solving local problem (\ref{simp_x}) using a (variable) number of local rounds of L-BFGS. Next, they broadcast the updated models to their neighbors (step~4). In case a neighbor is inactive, we assume the broadcast information can be stored in a buffer for the neighbor to access when becoming active again. Last, the dual variables are updated in step 5 based on the models received by the agents (notice that agent $i$ has received the needed information $\left\{x_j^{t+1}\right\}_{j \in\mathcal{N}_i}$ by step~4; in case a neighbor $j$ is inactive in this round, $x_j^{t+1}\equiv x_j^t $, so again this is available from past broadcasts).
\section{Convergence Analysis}
The analysis is under the following assumptions and the proofs are deferred to the appendix.
\newtheorem{assumption}{Assumption}
\begin{assumption}
Each local loss function is $L$-Lipschitz smooth, i.e., $\forall x, x'\in\mathbb{R}^d, i\in \left[m\right]$$$\norm{\nabla f_i(x)-\nabla f_i(x')}\leq L\norm{x-x'}.$$
\end{assumption}
\begin{assumption}
The local cost functions in (\ref{prob1}) are lower bounded, i.e., $\sum_{i=1}^m f_i(x_i)\geq F^\star$ for some $F^\star$ for all $\left\{x_i\right\}$.
\end{assumption}
We carry out the analysis with the following Lyapunov function 
\begin{align}
V^t=&\sum_{i=1}^m\norm{\nabla f_i(x_i^{t})+\sum_{j\in\mathcal{N}_i}y_{ij,i}^{t}}^2+\sum_{i=1}^m\sum_{j\in\mathcal{N}_i}\norm{x^t_i-z_{ij}^t}^2\notag\\=&\sum_{i=1}^m\norm{\nabla f_i(x_i^{t})+\phi_i^t}^2+\frac{1}{4}\sum_{i=1}^m\sum_{j\in\mathcal{N}_i}\norm{x_i^t-x_j^t}^2,\label{Vt}
\end{align}based on the equivalence of our method and (\ref{seem5}). Note that $V^t=0$ if and only if $\nabla f_i(x_i^{t})+\phi_i^t=0,\forall i\in [m]$ and $x_i^t-x_j^t=0,\forall (i,j)\in \mathcal{E}$, which means $x_1^t=\cdots=x_m^t$ (consensus) and $\sum_{i=1}^m\nabla f_i(x_i^{t})=0$ (this is because $y_{ij,i}^t+y_{ij,j}^t=0$ leads to  $\sum_{i=1}^m\phi_i^t=0$ for all $t$). This corresponds to a stationary point of (1).

By choosing $\mu_z > L$, in view of Assumption 1 it follows that the objective in (\ref{simp_x}) becomes strongly convex. Therefore, the analysis of \cite{liu1989limited} suggests that applying L-BFGS for (\ref{simp_x}) features linear convergence (with the number of iterations), and we use $r\in (0,1)$ to denote the rate parameter. We let $\tau$ denote the number of L-BFGS iterations in step 3 of Alg. 1 (this is allowed to be agent-specific, and we take the smallest number across users as $\tau$ for simplicity).
\newtheorem{remark}{Remark}
\newtheorem{theorem}{Theorem}
\begin{theorem}
Under Assumptions 1-2, by defining  $e^0=\left[\left(\nabla_{x_1} \mathcal{L}^0(x_1^0)\right)^\top,\hdots ,\left(\nabla_{x_m} \mathcal{L}^0(x_m^0)\right)^\top\right]^\top$, assuming that each agent is active with probability $p_i$ independently and $0~<~p_{\text{min}}:=\underset{i\in[m]}{\min} p_i$, choosing $\mu_z\geq1+2L, \mu_y\geq\frac{1152d_{\text{max}}^2\mu_z\lambda_{\text{max}}}{\lambda^2_{\text{min}}p_{\text{min}}}$ and $\tau$ such that $r^\tau\leq\frac{\lambda^2_{\text{min}}p_{\text{min}}}{4608d_{\text{max}}^4\mu_z\lambda_{\text{max}}}$,  the sequence generated by Algorithm~1 satisfies
\begin{align*}
\frac{1}{T}\sum_{t=0}^{T-1}\mathbb{E}\left[V^t\right]\leq\frac{1}{T}\left(C_1\mathbb{E}\left(\mathcal{L}^0- \mathcal{L}^T\right)+C_{2}\norm{e^0}^2\right),
\end{align*}where $\mathcal{L}^0=\mathcal{L}\left(x^0,y^0,z^0\right)$ and $\mathbb{E}\left[\mathcal{L}^T\right]$ is uniformly lower bounded. $C_1=\text{max}\left\{\frac{C_{8}}{C_3},\frac{C_7}{C_4}\right\}, C_{2}=C_6+C_1C_5$ are positive and $C_3$ to $C_{8}$ are as follows:
\begin{align*}
C_3=&\mu_z-2r^\tau\frac{\hat{C}_2}{\hat{C}_1}-\frac{2r^\tau\hat{C}_2\mu_y^2\mu_z^2}{\hat{C}_1(1-\hat{C}_4d_{\text{max}}^2)}\left(\frac{36\lambda_{\text{max}}}{\lambda^2_{\text{min}}\mu_y^2}+\hat{C}_4\right)\\&-\frac{\mu_z^2\mu_y}{1-\hat{C}_4d_{\text{max}}^2}\left(\hat{C}_4+\frac{36\lambda_{\text{max}}}{\lambda^2_{\text{min}}\mu_y^2}\right),\\
C_4=&p_{\text{min}}\frac{2\mu_z-1-2L}{4}-\frac{36r^\tau\hat{C}_2\lambda_{\text{max}}\mu_z^2(d_{\text{max}}^2\mu_z^2+L^2)}{\lambda^2_{\text{min}}\hat{C}_1(1-\hat{C}_4d_{\text{max}}^2)}\\&-\frac{18\lambda_{\text{max}}\mu_z^2(d_{\text{max}}^2\mu_z^2+L^2)}{\lambda^2_{\text{min}}\mu_y(1-\hat{C}_4d_{\text{max}}^2)},\\
C_5=&\frac{1}{\hat{C}_1}+\frac{12r^\tau\hat{C}_2\lambda_{\text{max}}(1+\hat{C}_3)}{\hat{C}_1^2\lambda^2_{\text{min}}(1-\hat{C}_4d_{\text{max}}^2)}+\frac{6\lambda_{\text{max}}(1+\hat{C}_3)}{\hat{C}_1\lambda^2_{\text{min}}\mu_y(1-\hat{C}_4d_{\text{max}})},\\
C_6=&\frac{4}{\hat{C}_1}+\frac{48r^\tau\hat{C}_2\lambda_{\text{max}}(1+\hat{C}_3)}{\hat{C}_1^2\lambda^2_{\text{min}}(1-\hat{C}_4d_{\text{max}}^2)}\\&+\frac{6\lambda_{\text{max}}(1+\hat{C}_3)(8\mu_zd_{\text{max}}+1)}{\hat{C}_1\lambda^2_{\text{min}}\mu_y(1-\hat{C}_4d_{\text{max}})},\\
C_7=&\frac{108r^\tau\hat{C}_2\lambda_{\text{max}}\mu_z^2(d_{\text{max}}^2\mu_z^2+L^2)}{\lambda^2_{\text{min}}\hat{C}_1(1-\hat{C}_4d_{\text{max}}^2)}+2L^2+8\mu_z^2d_{\text{max}}^2\\&+\frac{18\lambda_{\text{max}}\mu_z^2(d_{\text{max}}^2\mu_z^2+L^2)}{\lambda^2_{\text{min}}\mu_y(1-\hat{C}_4d_{\text{max}}^2)}\left(8\mu_z^2d_{\text{max}}^2+1\right),\\
C_{8}=&8r^\tau\frac{\hat{C}_2}{\hat{C}_1}+\frac{8r^\tau\hat{C}_2\mu_y^2\mu_z^2}{\hat{C}_1(1-\hat{C}_4d_{\text{max}}^2)}\left(\frac{36\lambda_{\text{max}}}{\lambda^2_{\text{min}}\mu_y^2}+\hat{C}_4\right)\\&+\frac{\mu_z^2(8\mu_z^2d_{\text{max}}+1)}{1-\hat{C}_4d_{\text{max}}^2}\left(\hat{C}_4+\frac{36\lambda_{\text{max}}}{\lambda^2_{\text{min}}\mu_y^2}\right),
\end{align*}where $\lambda_{\text{max}}, \lambda_{\text{min}}$ are the largest and the second smallest eigenvalues of the Laplacian respectively and $d_{\text{max}}$ is the largest degree and $\hat{C}_1=1-\frac{4r^\tau}{p_{\text{min}}},\hat{C}_2=\frac{4+3p_{\text{min}}^2-6p_{\text{min}}}{p_{\text{min}}^2}, \hat{C}_3=\frac{2}{p_{\text{min}}},\hat{C}_4=\frac{6\lambda_{\text{max}}}{\lambda^2_{\text{min}}}\left(\frac{r^\tau(4p_{\text{min}}-8r^\tau)(4+3p_{\text{min}}^2-6p_{\text{min}})}{p_{\text{min}}^2(p_{\text{min}}-4r^\tau)}+\frac{2+p_{\text{min}}^2-3p_{\text{min}}}{p_{\text{min}}\mu_y^2}\right)$.
\end{theorem}
The acceleration obtained by L-BFGS reflects in the local convergence rate $r$. Generally, L-BFGS has smaller $r$, which makes smaller $C_8, C_7$ and larger $C_3, C_4$ and thus result in smaller $C_1, C_2$, which means faster convergence. The following corollary serves to elucidate the dependencies of the convergence rate upon the conditioning of the local loss functions, the network topology, and activation probabilities.
\newtheorem{corollary}{Corollary}
\begin{corollary}
Following the conditions in Thm.~1 and further choosing $\mu_z = 2L+1, \mu_y=\frac{1152d_{\text{max}}^2\lambda_{\text{max}}\mu_z}{\lambda^2_{\text{min}}p_{\text{min}}}, \tau=\left\lceil\left(\ln{\frac{\lambda^2_{\text{min}}p_{\text{min}}}{4608d_{\text{max}}^4\lambda_{\text{max}}\mu_z}}\right)/\ln r\right\rceil$, it follows that $C_1$ and $C_{2}$ are of $\mathcal{O}\left(\frac{d_{\text{max}}^4L\lambda_{\text{max}}}{\lambda^2_{\text{min}}p_{\text{min}}}\right)$.
\end{corollary}
\section{experiments}
We evaluate our method on the problem of training a neural network with two fully connected layers with a total of $d=101632$ parameters. We consider a 10-class classification task and use data from the MNIST dataset \cite{lecun1998gradient}, which we distribute randomly across 20 agents. The communication topology is captured by a random graph: each edge is added independently with probability 0.2. In order to initialize the algorithm and simultaneously estimate the Lipschitz constant in Assumption~1, each agent undertakes 20 epochs of local training (on its local data) with a learning rate of 0.1, followed by 10 epochs with a learning rate of $10^{-7}$. The Lipschitz constant estimated by taking the maximum of the expression $\frac{\left \| \nabla f_i (x_i^{(r+1)}) - \nabla f_i (x_i^{(r)})\right\|}{\left \| x_i^{(r+1)} – x_i^{(r)}\right\|}$ over $r$, the counter for local training rounds. We compare our method with two popular baselines, namely gradient tracking (GT) \cite{alghunaim2022unified} and a primal-dual method called ADAPD \cite{mancino2023decentralized}. In addition, we also compare with a variant of our method that uses gradient descent as opposed to L-BFGS as the local solver (we term this as CADEN-GD). For all methods, the local work is set to 5 iterations, except for GT which is designed to use a single step of gradient descent. Since not all methods have dual variable (i.e., we can not compute $V^t$ for every method), we define the following relative error to capture the distance to stationary point $$\left\|\sum_{i=1}^m\nabla f_i(x_i) \right\|^2+\sum_{i=1}^{m-1}\left\|x_i-x_{i+1}\right\|^2.$$Fig.~1.a shows the decrease of the relative error with global round and Fig.~1.b represents the average test accuracy obtained by the local models. It becomes apparent that our method outperforms the baselines. In order to demonstrate the effect of varying workload, we also test a variant of our method that uses 5 local iterations in the first 100 rounds and a single one thereafter (we use -Red to denote this). It can be deduced from Fig.~2.a that such as choice is advantageous for CADEN, which is in support of computation savings as the algorithm progresses. In contrast, this is not the case when GD is used as the local solver (this further illustrates the merit of curvature acceleration). 
\begin{figure}[t]
  \centering
  \begin{subfigure}[b]{0.49\linewidth}
    \includegraphics[width=\textwidth]{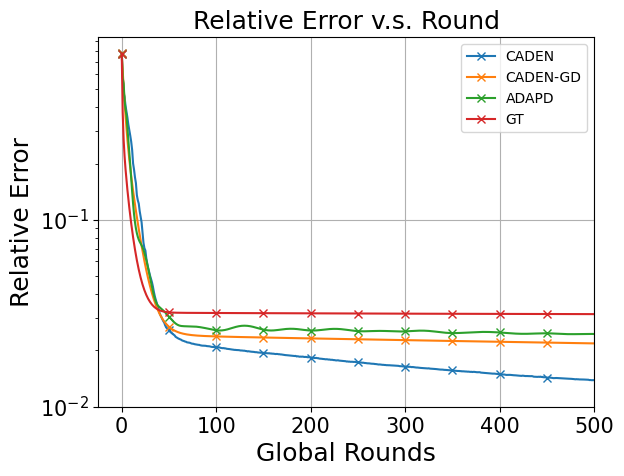}
    \caption{}
  \end{subfigure}
  \begin{subfigure}[b]{0.49\linewidth}
    \includegraphics[width=\textwidth]{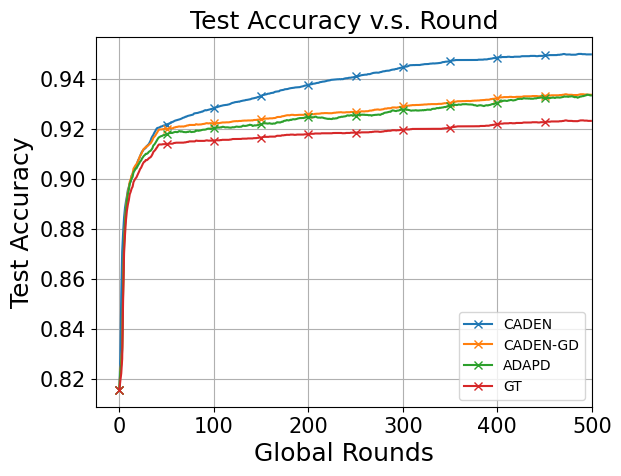}
    \caption{}
  \end{subfigure}\label{vs}
  \caption{Comparison on relative error (a) and test accuracy (b) with baselines.}
  \label{fig1}
  \begin{subfigure}[b]{0.49\linewidth}
    \includegraphics[width=\textwidth]{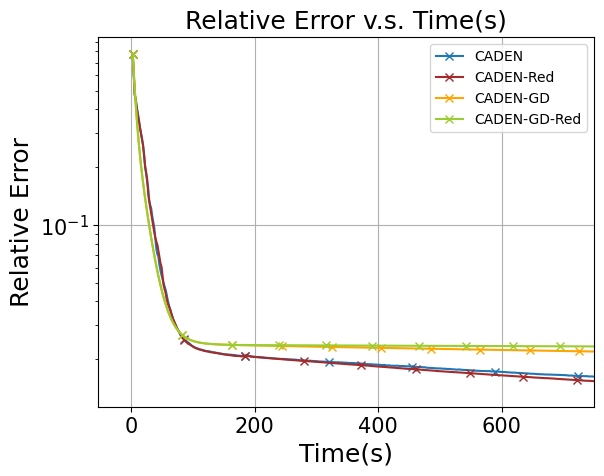}
    \caption{}
  \end{subfigure}
  \begin{subfigure}[b]{0.49\linewidth}
    \includegraphics[width=\textwidth]{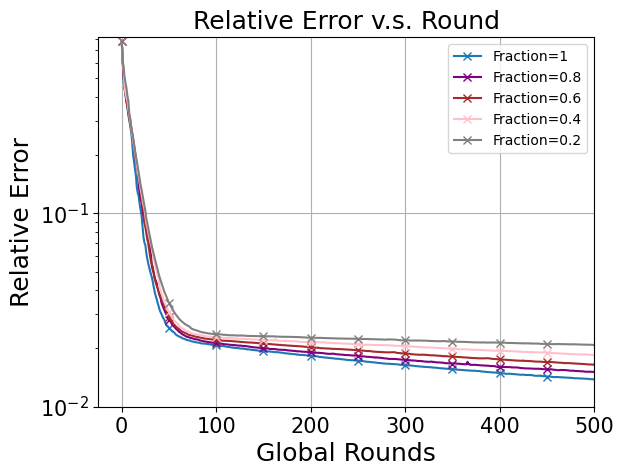}
    \caption{}
  \end{subfigure}
  \caption{Comparison with variants that decrease the number of local iterations after 100 global rounds (a) and for different levels of participation (b).}
  \label{fig2}
\end{figure}
\begin{table}[htbp]
\centering
\begin{adjustbox}{max width=\linewidth}
\begin{tabular}{lcccl}
\toprule
\textbf{Method} & \multicolumn{2}{c}{\textbf{Until 93\% test accuracy}} & & \textbf{Highest accuracy (\%)} \\
\cmidrule(r){2-3}
 & \textbf{Time (s)} & \textbf{Communications} & &  \\
\hline
GT & \texttimes & \texttimes & & 92.34 \\
ADAPD & 626.57 & 29,230 & & 93.35 \\
CADEN(-Red) & \textbf{228.57 (219.06)} & \textbf{8,658 (8,732)} & & \textbf{94.99 (94.82)} \\
CADEN-GD(-Red) & 544.41 (703.17) & 24,938 (33,744) & & 93.37 (93.04) \\
\bottomrule
\end{tabular}
\end{adjustbox}
\caption{Time and communication needed to reach a given test accuracy, and highest accuracy during 500 rounds. \texttimes \,\,in the first line signifies that GT failed to reach the target accuracy.}
\end{table}
Additionally, Table~1 provides detailed evaluation of the time and communication cost (transmission of one vector with dimension $d$ is counted to be 1) needed to reach a target accuracy. Once again, CADEN  depicts substantial savings: 59.8\% for computation time and 65.3\% for communication over baseline with the best performance. Moreover, CADEN-Red achieves even greater savings in computation time compared to CADEN. Finally, Fig. 2.b depicts our method in the presence of partial participation: as expected, higher participation results in faster convergence.

\section{Conclusion}
We have proposed CADEN, a primal-dual method for high-dimensional non-convex distributed optimization. The local problems are solved by L-BFGS, so as to accelerate the algorithm. Besides, partial participation and variable local work are allowed (this is vital for heterogeneous settings).  We have rigorously established sublinear convergence with a rate related to problem conditioning, network topology, and agent activation probabilities (Thm.~1 and Cor.~1). Last but not least, CADEN was shown experimentally to yield substantial communication and computation savings over baselines (Sec.~4).



\bibliographystyle{IEEE.bst}
\bibliography{ref.bib}
\section*{APPENDIX}
The following lemma establishes that consensus with each of the neighbors is equivalent to consensus with the average across neighbors.
\newtheorem{lemma}{Lemma}
\begin{lemma}
For $d_{\text{max}},\lambda_{\text{max}},\lambda_{\text{min}}$ as stated in Thm.~1, the following holds for $\lambda:=\frac{\lambda^2_{\text{min}}}{2\lambda_{\text{max}}}$
\begin{footnotesize}
\begin{align*}
\lambda\norm{Ax^t-Bz^t}^2\leq\sum_{i=1}^m\norm{\sum_{j\in\mathcal{N}_i}\left( x_i^t-z_{ij}^t\right)}^2\leq d_{\text{max}}\norm{Ax^t-Bz^t}^2.
\end{align*}
\end{footnotesize}
\end{lemma}
\begin{proof}
Since $z_{ij}^t = \frac{1}{2}\left(x_i^t+x_j^t\right)$, the first inequality is equivalent to $\frac{\lambda}{2}\norm{E_s^\top x^t}^2\leq\norm{\frac{1}{2}Lx^t}^2$ where $E_s$ is the signed graph incidence matrix (with a slight abuse of notation $L$, which means Laplacian here). For any $x$, we do the direct sum decomposition: $x = \bar{x}+\hat{x}$ such that $\bar{x}_1=\cdots = \bar{x}_m$ and $\bar{x}^\top\hat{x}=0$, then
\begin{align*}
    &\frac{\lambda}{2}\left\|E_s^\top x^t\right\|^2 = \frac{\lambda}{2}\left\|E_s^\top \hat{x}^t\right\|^2\leq \frac{\lambda}{2}\lambda_{\text{max}}\left\|\hat{x}^t\right\|^2\\\leq&\frac{\lambda^2_{\text{min}}}{4}\left\|\hat{x}^t\right\|^2\leq\left\|\frac{1}{2}L\hat{x}^t\right\|^2=\left\|\frac{1}{2}Lx^t\right\|^2.
\end{align*}
The second inequality holds because $\left\|\sum_{j\in\mathcal{N}_i}\left(x_i^t-z_{ij}^t\right)\right\|^2\leq d_i\sum_{j\in\mathcal{N}_i}\left\|x_{i}^t-z_{ij}^t\right\|^2$ where $d_i$ is the degree of agent~$i$.  
\end{proof}
\begin{proof}[\textbf{proof for Thm.1}]
The following proof can be roughly divided into twp parts: 1) we evaluate how much the Lagrangian decreases in one round; 2) we proceed to bound the optimality gap $V^t$. First, we have
\begin{align*}
&\mathcal{L}(x^{t+1},z^{t+1},y^{t+1})-\mathcal{L}(x^{t},z^{t},y^{t})\\=&\mathcal{L}(x^{t+1},z^{t+1},y^{t+1})-\mathcal{L}(x^{t+1},z^{t+1},y^{t})\\&+\mathcal{L}(x^{t+1},z^{t+1},y^{t})-\mathcal{L}(x^{t+1},z^{t},y^{t})\\&+\mathcal{L}(x^{t+1},z^{t},y^{t})-\mathcal{L}(x^{t},z^{t},y^{t}).
\end{align*} 
The first term is equal to $\left(y^{t+1}-y^t\right)^\top\left(Ax^{t+1}-Bz^{t+1}\right)=\frac{1}{\mu_y}\norm{y^{t+1}-y^t}^2_2=\mu_y\norm{Ax^{t+1}-Bz^{t+1}}^2$. The second term is equal to $-\mu_z\norm{z^t-z^{t+1}}^2$. By defining the local Lagrangian
\begin{footnotesize}
\begin{align*}
\mathcal{L}_i\left(x_i,y,z\right)=f_i(x_i)+\sum_{j\in\mathcal{N}_i}\left(y_{ij}^\top\left(x_i-z_{ij}\right)+\frac{\mu_z}{2}\left\|x_i-z_{ij}\right\|^2\right),
\end{align*}
\end{footnotesize}
and two local error terms $\hat{e}_i^t:=\nabla_{x_i}\mathcal{L}_i(\hat{x}_i^t,y^{t-1},z^{t-1})$ and $e_i^t:=\nabla_{x_i}\mathcal{L}_i(x_i^t,y^{t-1},z^{t-1})$, we have
\begin{align*}
&\mathcal{L}_i\left(\hat{x}_i^{t+1}, y^t, z^t\right)-\mathcal{L}_i\left(x_i^{t}, y^t, z^t\right)\\
\leq&\langle\nabla f_i(\hat{x}_i^{t+1}),\hat{x}_i^{t+1}-x_i^t\rangle+\frac{L}{2}\norm{\hat{x}_i^{t+1}-x_i^t}^2\\&+\sum_{j\in\mathcal{N}_i}\langle y_{ij,i}^t,\hat{x}_i^{t+1}-x_i^t\rangle\\&+\frac{\mu_z}{2}\langle \hat{x}_i^{t+1}+x_i^t-2z_{ij}^t,\hat{x}_i^{t+1}-x_i^t\rangle\\=&\langle\nabla f_i(\hat{x}_i^{t+1})+\sum_{j\in\mathcal{N}_i}y_{ij,i}^t+\mu_z(\hat{x}_i^{t+1}-z_{ij}^t),\hat{x}_i^{t+1}-x_i^t\rangle\\&+\frac{L}{2}\norm{\hat{x}_i^{t+1}-x_i^t}^2-\frac{\mu_z}{2}\sum_{j\in\mathcal{N}_i}\norm{\hat{x}_i^{t+1}-x_i^t}^2\\\leq&\norm{\hat{e}_i^{t+1}}^2+\frac{1}{4}\norm{\hat{x}_i^{t+1}-x_i^t}^2+\frac{L}{2}\norm{\hat{x}_i^{t+1}-x_i^t}^2\\&-\frac{\mu_z}{2}\sum_{j\in\mathcal{N}_i}\norm{\hat{x}_i^{t+1}-x_i^t}^2.
\end{align*}Here $\hat{x}^t_i$ is the virtual update variable, i.e., 
\begin{align*}
\hat{x}^{t}_i=\texttt{lbfgs}\,&\left(f_i(x_i)+\left(\phi_i^{t-1}\right)^{\top} x_i\right.\\&\left.+\frac{\mu_z}{2}\sum_{j\in\mathcal{N}_i}\norm{x_i-\frac{x_i^{t-1}+x_j^{t-1}}{2}}^2\right).
\end{align*}In other words, $\hat{x}_i^t$ is the local model agent~$i$ can obtain if it implements update in round $t-1$. 
Since agent~$i$ is active with probability $p_i$, we conclude that
\begin{align}
&\mathbb{E}^t\left[\mathcal{L}^{t+1}-\mathcal{L}^{t}\right]\leq\mu_y\norm{Ax^{t+1}-Bz^{t+1}}^2-\mu_z\norm{z^t-z^{t+1}}^2\notag\\&+\norm{\hat{e}^{t+1}}^2-p_{\text{min}}\left(\frac{\mu_z}{2}-\frac{1}{4}-\frac{L}{2}\right)\norm{\hat{x}^{t+1}-x^t}^2.\label{onestep}
\end{align}
Then, we bound $\left\|Ax^t-Bz^t\right\|^2$, from the definitions of $\hat{e}_i^{t+1}$ and $e_i^t$, we have 
 \begin{align*}
    &\hat{e}_i^{t+1}-e_i^t=\\&\nabla f_i(\hat{x}_i^{t+1})+\sum_{j\in\mathcal{N}_i}y_{ij,i}^t+\mu_z(\hat{x}_i^{t+1}-z_{ij}^t)\\&-\left(\nabla f_i(x_i^{t})+\sum_{j\in\mathcal{N}_i}y_{ij,i}^{t-1}+\mu_z(x_i^{t}-z_{ij}^{t-1})\right)\\=&\nabla f_i(\hat{x}_i^{t+1})-\nabla f_i(x_i^{t})+\sum_{j\in\mathcal{N}_i}\left(y_{ij,i}^{t}-y_{ij,i}^{t-1}+\right.\\&\left.\mu_z(\hat{x}_i^{t+1}-x_i^t-z_{ij}^t+z_{ij}^{t-1})\right)\\=&\nabla f_i(\hat{x}_i^{t+1})-\nabla f_i(x_i^{t})+\sum_{j\in\mathcal{N}_i}\left(\mu_y(x_i^{t}-z_{ij}^{t})+\right.\\&\left.\mu_z(\hat{x}_i^{t+1}-x_i^t-z_{ij}^t+z_{ij}^{t-1})\right).
\end{align*}By rearranging the terms, we have
\begin{align*}
&\mu_y^2\sum_i\norm{\sum_{j\in\mathcal{N}_i}x_i^{t}-z_{ij}^{t}}^2\leq 3\norm{e^t}^2+3\norm{\hat{e}^{t+1}}^2\\&+\left(9d_{\text{max}}^2\mu_z^2+9L^2\right)\norm{x^t-\hat{x}^{t+1}}^2+18\mu_z^2\norm{z^{t-1}-z^t}^2.
\end{align*}
By invoking Lemma~1 and adding from 1 to $T$ we get
\begin{align}
&\sum_{t=1}^T\norm{Ax^t-Bz^t}^2\leq\frac{2\lambda_{\text{max}}}{\lambda_{\text{min}}^2\mu_y^2}\sum_{t=1}^T\left(3\norm{e^t}^2+3\norm{\hat{e}^{t+1}}^2\right.\notag\\&\left.+18\mu_z^2\norm{z^{t-1}-z^t}^2+\left(9d_{\text{max}}^2\mu_z^2+9L^2\right)\norm{x^t-\hat{x}^{t+1}}^2\right).\label{sumx_z}
\end{align}
We now proceed to bound $e^t$ in relation to the local iteration number. From \cite{liu1989limited}, there exists some $r\in (0,1)$ such that
$$\norm{\nabla_{x_i} L_i\left(\hat{x}_i^{t+1},z^t,y^t\right)}^2\leq r^\tau \norm{\nabla_{x_i} L_i\left(x_i^{t},z^t,y^t\right)}^2.$$Since $\nabla_{x_i} L_i\left(x_i^{t},z^t,y^t\right)=\nabla_{x_i} L_i\left(x_i^{t},z^{t-1},y^{t-1}\right)+\\\sum_j\left(\left(y_{ij,i}^t-y_{ij,i}^{t-1}\right)+\mu_z\left(z_{ij}^{t-1}-z_{ij}^{t}\right)\right)$, by defining $\kappa_i^t := \sum_{j\in\mathcal{N}_i}\left(\left(y_{ij,i}^t-y_{ij,i}^{t-1}\right)+\mu_z\left(z_{ij}^{t-1}-z_{ij}^{t}\right)\right)$ we obtain
\begin{align}
&\norm{\hat{e}_i^{t+1}}^2\leq2r^\tau\left(\norm{e_i^{t}}^2+\left\|\kappa_i^t\right\|^2\right).\label{error_decay}
\end{align}Note that $e_i^t=\hat{e}_i^t$ with probability $p_i$, we have
\begin{align}
    &\mathbb{E}^{t-1}\left[\left\|e_i^t\right\|^2\right]\leq p_i\left\|\hat{e}_i^t\right\|^2+(1-p_i)(1+\xi)\left\|e_i^{t-1}\right\|^2\notag\\&+(1-p_i)(1+\frac{1}{\xi})\left\|\kappa_i^{t-1}\right\|^2\label{rela1}
\end{align}for any $\xi>0$. Substitute it into (\ref{error_decay}) and take expectation over all rounds, we have
\begin{align*}
    &\sum_{t=1}^T\mathbb{E}\left\|\hat{e}_i^{t+1}\right\|^2\leq \frac{2r^\tau p_i}{1-(1-p_i)(1+\xi)}\sum_{t=0}^T\mathbb{E}\left\|\hat{e}_i^{t}\right\|^2\\+&2r^\tau\left(1+\frac{(1-p_i)(1+\frac{1}{\xi})}{1-(1-p_i)(1+\xi)}\right)\sum_{t=1}^T\kappa_i^t.
\end{align*}Choosing large enough $\tau$ such that $\frac{2r^\tau p_i}{1-(1-p_i)(1+\xi)}<1$, we have
\begin{align}
\sum_{t=1}^T\mathbb{E}\left\|\hat{e}_i^{t+1}\right\|^2\leq\frac{1}{\hat{C}_1}\left\|e_i^0\right\|^2+2r^\tau\frac{\hat{C}_2}{\hat{C}_1}\sum_{t=1}^T\left\|\kappa_i^t\right\|^2,\label{sum_hat}
\end{align}where $\hat{C}_1=1-\frac{2r^\tau}{1-(1-p_{\text{min}})(1+\xi)},\hat{C}_2=1+\frac{(1-p_{\text{min}})(1+1/\xi)}{1-(1-p_{\text{min}})(1+\xi)}$. Then from (\ref{rela1}), we have
\begin{align*}
    &\sum_{t=1}^T\mathbb{E}\left\|e_i^t\right\|^2\leq p_i \sum_{t=1}^T\mathbb{E}\left\|\hat{e}_i^t\right\|^2\\+&(1-p_i)\sum_{t=1}^T\left((1+\xi)\mathbb{E}\left\|e_i^{t-1}\right\|^2+(1+\frac{1}{\xi})\left\|\kappa_i^{t-1}\right\|^2\right).
\end{align*}Choosing $\xi$ such that $\xi<\frac{p_{\text{min}}}{1-p_{\text{min}}}$, by defining $\hat{C}_3=\frac{1}{p_{\text{min}}+p_{\text{min}}\xi-\xi}$, we have
\begin{align}
    \sum_{t=1}^T\mathbb{E}\left\|e_i^t\right\|^2\leq \hat{C}_3 \sum_{t=1}^T\mathbb{E}\left\|\hat{e}_i^t\right\|^2+\frac{(1-p_i)(1+\xi)}{\xi}\sum_{t=0}^{T-1}\left\|\kappa_i^{t}\right\|^2.\label{sum_nohat}
\end{align}
Substitute (\ref{sum_hat}) and (\ref{sum_nohat}) into (\ref{sumx_z}), we obtain
\begin{align*}
    &\sum_{t=1}^T\mathbb{E}\left\|Ax^t-Bz^t\right\|^2\leq\frac{6\lambda_{\text{max}}}{\lambda_{\text{min}}^2\mu_y^2}\sum_{t=1}^T\left(6\mu_z^2\mathbb{E}\left\|z^t-z^{t-1}\right\|^2\right.\\&\left.+(3d_{\text{max}}^2\mu_z^2+3L^2)\mathbb{E}\left\|x^t-\hat{x}^{t+1}\right\|^2\right)+\\&(1+\hat{C}_3)\frac{6\lambda_{\text{max}}}{\lambda_{\text{min}}^2\mu_y^2}\left(\frac{1}{\hat{C}_1}\left\|e^0\right\|^2+2r^\tau\frac{\hat{C}_2}{\hat{C}_1}\sum_{t=1}^T\left\|\kappa^t\right\|^2\right)+\\&\frac{6\lambda_{\text{max}}}{\lambda_{\text{min}}^2\mu_y^2}\frac{(1-p_{\text{min}})(1+\xi)}{\xi}\sum_{t=0}^{T-1}\left\|\kappa^t\right\|^2.
\end{align*}Recall the definition of $\kappa_i^t$, we have
\begin{align*}
    &\sum_{t=1}^T\mathbb{E}\left\|Ax^t-Bz^t\right\|^2\leq \frac{6\lambda_{\text{max}}(1+\hat{C}_3)}{\hat{C}_1\lambda_{\text{min}}^2\mu_y^2}\left\|e^0\right\|^2\\&+\frac{36\lambda_{\text{max}}\mu_z^2}{\lambda_{\text{min}}^2\mu_y^2}\sum_{t=1}^T\mathbb{E}\left\|z^{t-1}-z^t\right\|^2\\&+\frac{18\lambda_{\text{max}}\mu_z^2}{\lambda_{\text{min}}^2\mu_y^2}(d_{\text{max}}^2\mu_z^2+L^2)\sum_{t=1}^T\mathbb{E}\left\|x^t-\hat{x}^{t+1}\right\|^2\\&+\hat{C}_4\sum_{t=1}^T\mathbb{E}\left(d_{\text{max}}^2\left\|Ax^t-Bz^t\right\|^2+\mu_z^2\left\|z^{t-1}-z^t\right\|^2\right),
\end{align*}where $\hat{C}_4=\frac{6\lambda_{\text{max}}}{\lambda_{\text{min}}^2\mu_y^2}\left(2r^\tau\mu_y^2\frac{\hat{C}_2(1+\hat{C}_1)}{\hat{C}_1}+\frac{(1-p_{\text{min}})(1+\xi)}{\xi}\right)$. By shifting the $\left\|Ax^t-Bz^t\right\|^2$-term on the right-hand side to the left and choosing suitable parameters such that $\hat{C}_4d_{\text{max}}^2<1$, we have
\begin{align}
    &\sum_{t=1}^T\mathbb{E}\left\|Ax^t-Bz^t\right\|^2\leq\frac{6\lambda_{\text{max}}(1+\hat{C}_3)}{\hat{C}_1\lambda_{\text{min}}^2\mu_y^2(1-\hat{C}_4d_{\text{max}}^2)}\left\|e^0\right\|^2\notag\\&+\frac{18\lambda_{\text{max}}\mu_z^2(d_{\text{max}}^2\mu_z^2+L^2)}{\lambda_{\text{min}}^2\mu_y^2(1-\hat{C}_4d_{\text{max}}^2)}\sum_{t=1}^T\mathbb{E}\left\|x^t-\hat{x}^{t+1}\right\|^2+\notag\\&\frac{\mu_z^2}{1-\hat{C}_4d_{\text{max}}^2}\left(\frac{36\lambda_{\text{max}}}{\lambda_{\text{min}}^2\mu_y^2}+\hat{C}_4\right)\sum_{t=1}^T\mathbb{E}\left\|z^{t-1}-z^t\right\|^2.\label{clear_sumx_z}
\end{align}
Substituting (\ref{clear_sumx_z}) into (\ref{sum_hat}), we have
\begin{align}
    &\sum_{t=1}^T\mathbb{E}\left\|\hat{e}^{t+1}\right\|^2\leq\frac{1}{\hat{C}_1}\left\|e^0\right\|^2+2r^\tau\frac{\hat{C}_2}{\hat{C}_1}\mu_z^2\sum_{t=1}^T\mathbb{E}\left\|z^{t-1}-z^t\right\|^2\notag\\&+\frac{12r^\tau\hat{C}_2\lambda_{\text{max}}(1+\hat{C}_3)}{\hat{C}_1^2\lambda_{\text{min}}^2(1-\hat{C}_4d_{\text{max}}^2)}\left\|e^0\right\|^2+\notag\\&\frac{36r^\tau\hat{C}_2\lambda_{\text{max}}\mu_z^2(d_{\text{max}}^2\mu_z^2+L^2)}{\lambda_{\text{min}}^2\hat{C}_1(1-\hat{C}_4d_{\text{max}}^2)}\sum_{t=1}^T\mathbb{E}\left\|x^t-\hat{x}^{t+1}\right\|^2+\notag\\&\frac{2r^\tau\hat{C}_2\mu_y^2\mu_z^2}{\hat{C}_1(1-\hat{C}_4d_{\text{max}}^2)}\left(\frac{36\lambda_{\text{max}}}{\lambda_{\text{min}}^2\mu_y^2}+\hat{C}_4\right)\sum_{t=1}^T\mathbb{E}\left\|z^{t-1}-z^t\right\|^2.\label{clear_hat}
\end{align}
Now we choose $\xi=\frac{p_{\text{min}}}{2(1-p_{\text{min}})}$ (without loss of generality, here we assume that $p_{\text{min}}<1$; if $p_{\text{min}}=1$, what in Thm.~1 still holds since $1-p_{\text{min}}$ will not appear in denominators of parameters by abbreviation), by adding (\ref{onestep}) from $0$ to $T-1$ and substituting (\ref{clear_hat}) further gives
\begin{align}
\frac{1}{T}\sum_{t=0}^T C_3\mathbb{E}\norm{z^t-z^{t-1}}^2+C_4\mathbb{E}\norm{\hat{x}^{t}-x^{t-1}}^2\notag\\\leq\frac{1}{T}\mathbb{E}\left(\mathcal{L}^0-\mathcal{L}^{T}\right)+\frac{1}{T}C_5\mathbb{E}\norm{e^0}^2,\label{result1}
\end{align}where $C_3, C_4, C_5$ are as stated in Thm.~1
We continue by bounding the optimality gap $V^t$. First, we have
\begin{align*}
    &\sum_{i=1}^m\left\|\nabla f_i(x_i^t)+\phi_i^t\right\|^2\\=&\sum_{i=1}^m\left\|\nabla f_i(\hat{x}_i^{t+1})-\nabla f_i(\hat{x}_i^{t+1})+\nabla f_i(x_i^t)+\phi_i^t\right\|^2\\\leq&2L^2\left\|\hat{x}^{t+1}-x^t\right\|^2+2\sum_{i=1}^m\left\|\nabla f_i(\hat{x}_i^{t+1})+\phi_i^t\right\|^2\\\leq&2L^2\left\|\hat{x}^{t+1}-x^t\right\|^2+4\left\|\hat{e}^{t+1}\right\|^2+8\mu_z^2d_{\text{max}}^2\left\|\hat{x}^{t+1}-x^t\right\|^2\\&+8\sum_{i=1}^m\left\|\mu_z\sum_{j\in\mathcal{N}_i}x_i^{t}-z_{ij}^t\right\|^2.
\end{align*}So for $V^t$, we have
\begin{align*}
    V^t\leq& 4\left\|\hat{e}^{t+1}\right\|^2+\left(2L^2+8\mu_z^2d_{\text{max}}^2\right)\left\|\hat{x}^{t+1}-x^{t}\right\|^2\\&+\left(8\mu_z^2d_{\text{max}}+1\right)\left\|Ax^t-Bz^t\right\|^2.
\end{align*}Again, adding from $0$ to $T-1$, taking expectation and substituting (\ref{sumx_z}) into it, we have
\begin{align}
\sum_{t=0}^{T-1}\mathbb{E}\left[V^t\right]\leq& C_6\norm{e^0}^2+\sum_{t=0}^{T-1}C_7\mathbb{E
}\norm{\hat{x}^{t+1}-x^t}^2\notag\\+&\sum_{t=0}^{T-1}C_8\mathbb{E}\norm{z^{t-1}-z^t}^2\label{result2}
\end{align}
where $C_6, C_7, C_8$ are as stated in Thm.~1. Compare (\ref{result1}) and (\ref{result2}), we have
\begin{align*}
\frac{1}{T}\sum_{t=0}^{T-1}\mathbb{E}\left[V^t\right]\leq&\frac{1}{T}\text{max}\left\{\frac{C_8}{C_3},\frac{C_7}{C_4}\right\}\mathbb{E}\left(\mathcal{L}^0-{\mathcal{L}}^{T}\right)\\&+\frac{1}{T}\left(C_6+\text{max}\left\{\frac{C_8}{C_3},\frac{C_7}{C_4}\right\}C_5\right)\norm{e^0}^2
\end{align*}
Finally, we prove the that $\mathbb{E}\left[\mathcal{L}^T\right]$ is uniformly lower bounded. We have
\begin{align*}
    &\mathcal{L}^t-F^\star\geq (y^t)^\top(Ax^t-Bz^t)+\frac{\mu_z}{2}\left\|Ax^t-Bz^t\right\|^2\\\geq& (y^t)^\top(Ax^t-Bz^t)=\frac{1}{\mu_y}\left(y^t\right)^\top\left(y^t-y^{t-1}\right)\\=&\frac{1}{\mu_y}\left(\left\|y^t\right\|^2-\left\|y^{t-1}\right\|^2+\left\|y^t-y^{t-1}\right\|^2\right)\\\geq&\frac{1}{\mu_y}(\left\|y^t\right\|^2-\left\|y^{t-1}\right\|^2).
\end{align*}
Therefore, $\sum_{t=1}^T\left(\mathcal{L}^t-F^\star\right)\geq -\frac{1}{\mu_y}\left\|y^0\right\|^2$ for all $T$. Form (\ref{result1}), we have
\begin{align*}
    \mathbb{E}\left[\mathcal{L}^T\right]\leq&\mathcal{L}^0+C_5\left\|e^0\right\|^2\\&-\sum_{t=0}^T\left(C_3\mathbb{E}\left\|z^t-z^{t-1}\right\|^2+C_4\mathbb{E}\left\|\hat{x}^t-x^{t-1}\right\|^2\right),
\end{align*}which means $\sum_{t=0}^T\left(\mathbb{E}\left\|z^t-z^{t-1}\right\|^2+\mathbb{E}\left\|\hat{x}^t-x^{t-1}\right\|^2\right)$ is uniformly upper bounded over $T$. Then (\ref{onestep}) together with (\ref{clear_sumx_z}) and (\ref{clear_hat}) yield that $\mathbb{E}\left[\mathcal{L}^{T_1}-\mathcal{L}^{T_2}\right]$ is uniformly upper bounded for all $T_1\geq T_2$. Therefore, one can obtain that $\mathbb{E}\left[\mathcal{L}^T\right]$ is lower bounded by contradiction and we finish the proof.
\end{proof}
\end{document}